\documentclass[11pt, reqno]{amsart}
\allowdisplaybreaks[1]

\usepackage{relsize}
\usepackage{amsmath,amsfonts,amssymb,amsthm,mathrsfs}
\usepackage{hyperref}
\usepackage{cleveref}
\usepackage[margin=2.5cm]{geometry}
\usepackage{setspace}
\usepackage{mathtools}

\usepackage{enumitem}

\usepackage{setspace}
\usepackage{xcolor}
\keywords{integration theory, quantum instruments, completely positive maps, non-atomic instruments, Lyapunov theorem, Krein-Milman theorem, $C^*$-convexity}
\subjclass[2020]{46B22,46L51, 46G10,81P15.}

\allowdisplaybreaks[1]

\newcommand{\bh}{\mathcal{B}(\mathcal{H})}
\newcommand{\bk}{\mathcal{B}(\mathcal{K})}

\newcommand{\cbxa}{C_b(X,\mathcal{A})}
\newcommand{\ox}{\mathcal{O}(X)}

\newcommand{\be}{\begin{equation}}
	\newcommand{\ee}{\end{equation}}
\newcommand{\bea}{\begin{eqnarray}}
	\newcommand{\eea}{\end{eqnarray}}
\newcommand{\bean}{\begin{eqnarray*}}
	\newcommand{\eean}{\end{eqnarray*}}
\newcommand{\brray}{\begin{array}}
	\newcommand{\erray}{\end{array}}

\newtheorem{dfn}{Definition}[section]
\newtheorem{thm}[dfn]{Theorem}
\newtheorem{lmma}[dfn]{Lemma}
\newtheorem{ppsn}[dfn]{Proposition}
\newtheorem{crlre}[dfn]{Corollary}
\newtheorem{xmpl}[dfn]{Example}
\newtheorem{rmrk}[dfn]{Remark}

\theoremstyle{definition}\newtheorem*{notation}{Notation}
\theoremstyle{definition}

\newcommand{\cst}{C^*}

\newcommand{\bdfn}{\begin{dfn}\rm}
	\newcommand{\bthm}{\begin{thm}}
		\newcommand{\blmma}{\begin{lmma}}
			\newcommand{\bppsn}{\begin{ppsn}}
				\newcommand{\bcrlre}{\begin{crlre}}
					\newcommand{\bxmpl}{\begin{xmpl}}
						\newcommand{\brmrk}{\begin{rmrk}\rm}
							
							\newcommand{\edfn}{\end{dfn}}
						\newcommand{\ethm}{\end{thm}}
					\newcommand{\elmma}{\end{lmma}}
				\newcommand{\eppsn}{\end{ppsn}}
			\newcommand{\ecrlre}{\end{crlre}}
		\newcommand{\exmpl}{\end{xmpl}}
	\newcommand{\ermrk}{\end{rmrk}}

\newcommand{\bbc}{\mathbb{C}}

\newcommand{\bbn}{\mathbb{N}}

\newcommand{\cla}{\mathcal{A}}
\newcommand{\clb}{\mathcal{B}}

\newcommand{\clh}{\mathcal{H}}
\newcommand{\cli}{\mathcal{I}}
\newcommand{\clj}{\mathcal{J}}
\newcommand{\clk}{\mathcal{K}}

\newcommand{\cls}{\mathcal{S}}

\newcommand{\clo}{\mathcal{O}}

\makeatletter
\let\@wraptoccontribs\wraptoccontribs
\makeatother

\title[Integration theory]{Integration Theory for Completely positive Instruments: A Lyapunov-Type Theorem and Applications}

\author{Arghya Chongdar}
\address{Indian Statistical Institute, Stat Math Unit, R V College Post, Bengaluru, 560059, India.}
\email{chongdararghya@gmail.com}

\date{}
\begin{document}
	\begin{abstract}

We develop a theory of integration with respect to quantum instruments through two complementary approaches: a vector measure formulation based on Bartle's integration theory and a tensor product construction. As a principal application, we establish a CP map-valued Lyapunov theorem by characterizing the convexity of the range of non-atomic completely positive instruments via the associated integration map. This extends the work of Plosker and Ramsey~\cite{Plosker_Ramsey} from POVMs to completely positive instruments and gives a general characterization of the phenomenon exhibited in their setting. We also establish a correspondence between completely positive instruments and completely positive maps, prove a Krein--Milman type theorem for the $C^*$-convex set of unital completely positive instruments, and show that the tensor product construction connects the theory of CP instruments with the integration theory for POVMs due to Farenick \emph{et al}(\cite{douglus_plosker_ramsey_povmintegration_1}).

\end{abstract}

\maketitle
\section{Introduction}

Integration with respect to vector measures occupies a central position in modern analysis, providing a unified framework that connects functional analysis, operator theory, Banach space geometry, and convexity theory. The classical theories of Bochner, Pettis, and Dunford integration were unified by Bartle's integration theory for vector measures \cite{bartle_integral}, which has since become a fundamental tool in the study of vector-valued functions (see \cite{uhl_vectormeasures} for a comprehensive account). Beyond its intrinsic analytic interest, vector integration has led to several remarkable applications, among which Lyapunov's convexity theorem \cite{liapounoff1940fonctions} occupies a distinguished place. One of the striking features of Lyapunov's theorem is that the convexity of the range of a non-atomic vector measure is intimately connected with the associated integration operator, revealing a deep interplay between integration and the geometry of vector measures.

A natural question is whether an analogous integration theory exists in the noncommutative setting. Completely positive (CP) instruments, introduced by Davies and Lewis \cite{davies} and subsequently refined by Ozawa \cite{ozawa}, provide the mathematical framework for describing both the statistics of quantum measurements and the corresponding post-measurement state transformations. Consequently, CP instruments have become fundamental objects in quantum information theory, operator algebras, and quantum probability. Their convex structure has been extensively investigated from the viewpoint of classical convexity
\cite{holevo,pellonapaa1,pellonapaa2,Srinivas}, while a genuinely noncommutative $C^*$-convex framework has recently been developed in \cite{bhat_chongdar_sruthymurali}. Nevertheless, despite the central role played by integration in the classical theory of vector measures, no comparable integration theory is presently available for quantum instruments.

The purpose of this paper is to develop an integration theory for completely positive instruments that plays the same foundational role as Bartle's integration theory for vector measures. We present two complementary constructions. The first extends Bartle's vector integration to the setting of completely positive instruments by introducing suitable notions of measurability and integrability for $\mathcal A$-valued functions. The second is based on tensor products and provides an intrinsic operator-algebraic formulation of integration, naturally accommodating a broader class of measurable functions. Together, these two approaches establish a unified framework that reveals new connections between quantum instruments, completely positive maps, and operator-valued measures.

A principal application of this framework is an operator-valued analogue of Lyapunov's convexity theorem. We characterize the convexity of the range of a non-atomic completely positive instrument in terms of the non-injectivity of its associated integration map, thereby extending the operator-valued Lyapunov theorem of Plosker and Ramsey \cite{Plosker_Ramsey} from positive operator-valued measures to completely positive instruments. In particular, our characterization resolves a question raised in Example~3.4 of \cite{Plosker_Ramsey}. Beyond the Lyapunov theorem, the integration theory developed here leads naturally to two further structural results. First, we establish a correspondence between regular completely positive instruments and completely positive maps, extending both the classical Riesz--Markov representation theorem and the correspondence between POVMs and completely positive maps. Secondly, we prove a Krein--Milman type theorem for the $C^*$-convex set of unital completely positive instruments, thereby providing a structural description of this class in terms of its $C^*$-extreme points.

The paper is organized as follows.

In Section~\ref{basic terminologies}, we introduce completely positive instruments on measurable spaces, fix the notation used throughout the paper, and recall the basic examples and structural results, including the bi-dilation theorem established in \cite{bhat_chongdar_sruthymurali}.

The main body of the paper begins with Section~\ref{Classical approach}, where we develop integration with respect to instruments from the perspective of classical vector measure theory. We first establish the connection between instruments and $\mathcal A^*$-valued vector measures, showing that every instrument naturally gives rise to a family of vector measures of bounded variation (Remark~\ref{reformulation of definition of instruments}). We then introduce a total variation metric on the space of instruments and prove its completeness (Theorem~\ref{completeness w.r.t TV metric}). Next, we develop an integration theory for $\mathcal A$-valued functions by introducing weak and strong measurability and proving a Pettis-type theorem (Theorem~\ref{petti-type theorem}), which allows us to define integration for bounded separably valued measurable functions.

The remainder of Section~\ref{Classical approach} is devoted to three applications of the integration theory. In Subsection~\ref{Application 1: Lyapunov-type theorem}, we establish a Lyapunov-type theorem for non-atomic instruments by characterizing the non-injectivity of the associated integration map (Theorem~\ref{rangecharacterization}), and we illustrate the necessity of this condition through an infinite-dimensional example (Example~\ref{injective example}). Subsection~\ref{Application 2: Correspondence between instruments and CP maps} establishes a correspondence between regular completely positive instruments and completely positive maps, extending both the classical Riesz--Markov representation theorem and the CP--POVM correspondence. In Subsection~\ref{Application 3: Krein-Milman type theorem}, we introduce the bounded-weak topology on the space of instruments (Definition~\ref{definition of bw topology on instruments}) and prove a Krein--Milman type theorem (Theorem~\ref{Krein-Milman type theorem for instruments}), showing that every unital completely positive instrument is a BW-limit of $C^*$-convex combinations of $C^*$-extreme instruments.

The final section presents an alternative integration theory based on tensor-product techniques. This approach is intrinsically operator-algebraic and naturally extends the class of integrable functions beyond those obtained through Bartle-type approximation. We conclude with Example~\ref{non_trivial_instrument}, which demonstrates the advantage of the tensor-product construction and shows that it naturally connects the integration theory for POVMs developed by Farenick \emph{et al.} with instruments

  For background on completely positive maps and positive operator-valued measures, we refer the reader to \cite{arvesonsubalgebra,pisiernormal,daviesquantum,paulsen_book} and the references therein. Throughout the paper, we use the abbreviations UCP for \emph{unital completely positive} and POVM for \emph{positive operator-valued measure}. Unless stated otherwise, all Hilbert spaces are assumed to be complex, with inner products anti-linear in the first variable. For a Hilbert space $\clh$, we denote by $\bh$ the $C^*$-algebra of all bounded linear operators on $\clh$, and by $I_{\clh}$ the identity operator on $\clh$.
 
\section{Basic terminologies}\label{basic terminologies}
We begin by recalling the notion of completely positive (CP) instruments introduced in \cite{bhat_chongdar_sruthymurali}, together with the basic notation and structural results that will be used throughout the paper.

\begin{dfn}[Quantum instruments]\label{maindefinition}\index{quantum instruments}
	Let $(X,\ox)$ be a measurable space, $\mathcal{A}$ be a unital $C^*$-algebra and $\mathcal{H}$ be a Hilbert space. A  \emph{CP instrument} on $(X,\clo(X))$ with values in $CP(\cla,\bh)$ is a map $\cli: \ox\to CP(\cla,\bh)$ satisfies:  
	\begin{enumerate}
		\item $\cli(A)\in CP(\cla,\bh)~,$ for all~$A\in~\mathcal{O}(X);$ 
		\item  for every $h,k\in\mathcal{H}$ and any $a\in\mathcal{A}$, the map $\cli_{a,h,k}:\mathcal{O}(X)\to\mathbb{C}$ defined  by
		\begin{equation}
			\cli_{a,h,k}(A)=\langle h,\cli(A)(a)k\rangle~~\text{for all } A\in\ox,\label{eq:notation for cli_a,h,k}
		\end{equation}
		is a  complex measure.
	\end{enumerate}
	
	Moreover, a quantum instrument $\cli$ is said to be 
	\begin{itemize}
		\item \emph{normalized} or \emph{unital}\index{normalized instrumnet} if $\cli(X)(1_\cla)=1_\mathcal{H}$ i.e. if the completely positive map $\cli(X)$ is \emph{unital}.
		\item \emph{normal}\index{normal! instrument} if $\cla$ is a von Neumann algebra and $\cli(A)$'s are normal completely positive map for all $A\in\ox.$
		\item \emph{spectral}\index{spectral instrument} if $\cli(A)$ is a $*$-homomorphism for all $A\in \ox$ and is unital.
	\end{itemize}
\end{dfn}

One can easily check that POVMs and CP maps arise naturally as special cases of CP instruments.

\begin{rmrk}\label{complex measures associated to an instrument}
	For any CP instrument $\cli$, we denote by $\cli_{a,h,k}$ the associated complex measure defined in \eqref{eq:notation for cli_a,h,k}.
	It is straightforward to verify that a CP instrument $\cli$ is completely determined by the family of complex measures $\{ \mathcal{I}_{a,h,k} : a \in \mathcal{A},\ h,k \in \mathcal{H} \}.$
	
\end{rmrk}
\begin{thm}(Bi-dilation Theorem)\label{Bi-dilation Theorem}(Theorem 2.7,\cite{bhat_chongdar_sruthymurali})
    Let $\cli: \ox \to CP(\cla,\bh)$ be an instrument. Then there exists a quadruple $(\clk,\pi,E,V),$ where $\clk$ is a Hilbert space, $\pi: \cla \to \bk$ is a unital $*$-homomorphism, $E: \ox \to \bk$ is a spectral measure, and $V\in\clb(\clh,\clk)$ such that, \begin{equation}\label{naimarkdilation}
			\cli(A)(a)=V^*\pi(a)E(A)V\quad\text{and}\quad E(A)\pi(a)=\pi(a)E(A),
			\forall a \in \cla, A \in \ox
		\end{equation}
		and satisfies the minimality condition: $[\pi(\cla)E(\ox)V\clh]=\clk.$ Such a quadruple is unique up to unitary equivalence and is called a minimal bi-dilation tuple.
\end{thm}
\begin{notation}
We write $I_{\clh}(X,\cla)$ for the collection of all unital completely positive instruments and $Ins_{\clh}(X,\cla)$ for the collection of all completely positive instruments on $(X,\ox)$ with values in $CP(\cla,\bh)$.
\end{notation}
\section{Classical approach}\label{Classical approach}
To formally introduce the notion of integration with respect to an instrument, in this approach, we draw upon the theory of the Bartle integral (see \cite{bartle_integral}). In order to establish the necessary theoretical foundation, we first recall the definition of vector-valued measures taking values in a Banach space and review some of their fundamental properties. For a comprehensive exposition on vector measure theory, see \cite{ryan_tensorproduct_book}.

\color{black}
\begin{dfn}[Vector valued measures]\index{Vector valued measures}
	A vector measure on a measurable space $(X,\clo(X))$ is a countably additive function $\mu$ on $\clo(X)$ with values in a Banach space $\clb$. Thus if $\{E_n\}$ is a sequence of mutually disjoint measurable subsets of $X$ with $\cup_nE_n=E$, then the series $\sum_n \mu(E_n)$ converges to $\mu(E)$ unconditionally.
\end{dfn}
\begin{rmrk} In finite-dimensional Banach spaces, absolute convergence and unconditional convergence are equivalent. However, in the infinite-dimensional setting, unconditional convergence does not necessarily imply absolute convergence.
\end{rmrk}
Analogous to the case of complex measures, the notions of semi-variation and total variation are equally applicable in this context.

\begin{dfn}[Semi-variation and Total variation]
	Let $\mu$ be a vector measure on $(X,\clo(X))$ with values in a Banach space $\clb$. The semi-variation of $\mu$ is the extended non-negative function $\|\mu\|$ whose value on a set $A\in\ox,$ denoted by $\|\mu\|(A)$ or $\|A\|,$ is defined to be 
	\begin{align*}
		\|A\|=\sup \{ |\sum_i f(\mu(E_i))|: \{E_1,\cdots,E_n\} \text{ is a partition of $A$ and} ~f \in \clb^* ~\text{with}~ \|f\|\leq1\}
	\end{align*}
	and the total variation of $\mu,$ the extended non-negative function $|\mu|,$ is defined to be
	\begin{align*}
		|\mu|(A)=\sup \{\sum_i\lVert\mu(E_i)\rVert: \{E_1,\cdots,E_n\} \text{ is a partition of $A$}\}.
	\end{align*}
\end{dfn}
\begin{rmrk}
	Unlike complex measures, the total variation $|\mu|$ of a vector measure $\mu$ is not necessarily finite.
\end{rmrk} 
\begin{dfn}[Bounded variation]
	$\mu$ is said to be of bounded variation if $|\mu|(X)<\infty$, i.e $|\mu| $ is a finite positive measure on $(X,\clo(X))$.
\end{dfn}

As noted in Remark \ref{complex measures associated to an instrument}, the complex measures associated with an instrument can completely determine the instrument. However, it is also possible to associate a broader class of vector measures to an instrument, extending beyond the scope of complex measures. These vector measures offer a more general framework for understanding and analyzing the structure and behavior of instruments.

\begin{ppsn}\label{measurecorrespondence}
	Let $\cli:\ox\to CP(\cla,\bh)$ be a CP instrument and let $\rho$ be a density operator on the separable Hilbert space $\mathcal{H}.$ Then the map $\cli_\rho:\ox\to\cla^*$ defines a $\cla^*$-valued vector measure on $(X,\clo(X))$ which are of bounded variation, given by
	$\cli_\rho(A)(a)=Tr(\cli(A,a)\rho),$ where $Tr$ is the usual trace functional on $\bh.$ 
\end{ppsn}
\begin{proof}
	Observe that, the Equation \eqref{eq:notation for cli_a,h,k} implies that for any $a\in\cla$ and countable collection of disjoint measurable sets $\{A_i\}_{i=1}^\infty$ in $\ox$ the series $\sum_{i=1}^\infty \cli(A_i,a)$ converges in WOT. Since on the bounded subsets of $\bh,$ WOT and $\sigma$-weak topology both agree, therefore, for any state $\rho\in\bh,$
	\begin{equation}\label{sigma weak convergence of instruments}
		Tr(\sum_{i=1}^\infty\cli(A_i,a)\rho)=\sum_{i=1}^\infty Tr(\cli(A_i,a)\rho).
	\end{equation}
	It's clear from the definition of $\cli_\rho$ that for every $A\in\ox,~\cli_\rho$ defines a positive linear functional on the $C^*$-algebra $\cla.$
	Let us prove the countable additivity of $\cli_\rho$. Let $\{A_i\}_{i=1}^\infty$ be a countable family of disjoint measurable subsets of $X$. We have to prove that,
	\begin{eqnarray}\label{countableadditivity}
		\cli _\rho(\cup_{i=1}^\infty A_i)=\sum_{i=1}^\infty \cli_\rho(A_i),
	\end{eqnarray}
	where the right-hand side of the Equation \ref{countableadditivity} converges unconditionally in the norm. Observe that,
	\begin{eqnarray}\nonumber
		\sum_ {i=1}^\infty \lVert \cli_\rho(A_i)\rVert &=&	\sum_ {i=1}^\infty\cli _\rho(A_i)(1_\cla)\quad (\text{since, $\cli_\rho$ is a positive linear functional on $\cla$})\\\nonumber
		&=& \sum_ {i=1}^\infty Tr(\cli(A_i,1_\cla)\rho)\\\nonumber
		&=&Tr(\cli(A,1_\cla)\rho)\quad(\text{follows from the Equation \ref{sigma weak convergence of instruments}})\\\nonumber
		&=& \cli _\rho(A)(1_\cla)\\\label{cauchycriterion}
		&<&\infty ,
	\end{eqnarray} 
	proving that the series in Equation \ref{countableadditivity} is a Cauchy series. Also for each $a \in \cla$ we have,
	\begin{eqnarray}\nonumber
		\cli_\rho(\cup_{i=1} ^\infty A_i)(a)
		&=& Tr(\cli(\cup_{i=1} ^\infty A_i,a)\rho)\\\nonumber
		&=&\sum_{i=1 }^\infty Tr(\cli(A_i,a)\rho)\\\
		&=& \sum_{i=1 }^\infty\cli _\rho(A_i)(a)\label{pointwiseconvergence}
	\end{eqnarray}
	establishing the point-wise convergence of the series in Equation \ref{countableadditivity}. Combining Equation \ref{cauchycriterion} and Equation \ref{pointwiseconvergence}, we can conclude that the series in Equation \ref{countableadditivity} converges unconditionally in the norm. It remains to prove that $\cli_\rho$ is of bounded variation. We have,
	\begin{eqnarray}\nonumber
		|\cli _\rho|(X)&= &\sup \{\sum_i\lVert\cli_\rho(E_i)\rVert: \{E_1,\cdots,E_n\} \text{ is a partition of X}\}\\\nonumber
		&=& \sup\{\sum_i\cli_\rho(E_i)(1_\cla):\{E_1,\cdots,E_n\} \text{ is a partition of X}\} \\\nonumber
		&=& \sup \{\sum_ i Tr(\cli(E_i,1_\cla)\rho): \{E_1,\cdots,E_n\} \text{ is a partition of X}\}\\\nonumber
		&=& Tr(\cli(X,1_\cla)\rho)\\\nonumber
		&=& \cli_\rho(X)(1_\cla)\\\nonumber
		&<& \infty,
	\end{eqnarray}
	finishing the proof. 
\end{proof}
\begin{rmrk}\label{finite measure realization for faithful states}
Whenever the Hilbert space $\clh$ is separable, faithful density operators on $\clh$ always exist. If $\sigma\in\bh$ is any faithful density operator, then the associated scalar measure $|\cli_\sigma|$ is a finite measure on $(X,\ox)$. Moreover, $\cli$ and $|\cli_\sigma|$ are mutually absolutely continuous; that is,
\[
\cli\ll |\cli_\sigma|
\quad\text{and}\quad
|\cli_\sigma|\ll\cli.
\]
In particular, if $\sigma_1$ and $\sigma_2$ are two faithful density operators on $\clh$, then the corresponding finite measures $|\cli_{\sigma_1}|$ and $|\cli_{\sigma_2}|$ are mutually absolutely continuous.
\end{rmrk}

\begin{rmrk}\label{vector measures associated to instruments}
More generally, for every trace class operator $\tau\in\bh$, define
\[
\cli_\tau(A)(a):=\operatorname{Tr}(\cli(A,a)\tau),
\qquad A\in\clo(X), a\in\cla.
\]
Since every trace class operator is a linear combination of density operators, $\cli_\tau$ is a $\cla^*$-valued measure of bounded variation.

In particular, for $h,k\in\clh$, taking $\tau=|k\rangle\langle h|$ yields the $\cla^*$-valued measures
\[
\cli_{h,k}(A)(a)=\langle h,\cli(A,a)k\rangle,
\]
which completely determine the instrument $\cli$, analogous to the complex measure construction in Remark~\ref{complex measures associated to an instrument}. Moreover,
\[
|\cli_{h,k}|(X)\le \|h\|\|k\|\|\cli(X,1_\cla)\|,
\]
so each $\cli_{h,k}$ is of bounded variation.
\end{rmrk}

\begin{rmrk}\label{reformulation of definition of instruments}
To conclude that $\cli_{h,k}$ is a $\cla^*$-valued measure for every
$h,k\in\clh$, the separability of $\clh$ is not required. Indeed, using
the same argument as in the proof of Theorem~\ref{measurecorrespondence},
one first establishes that $\cli_{h,h}$ is a $\cla^*$-valued measure for
every $h\in\clh$. The general case then follows from the polarization
identity
\[
\cli_{h,k}
=
\frac14\sum_{l=0}^{3}
i^{-l}\,
\cli_{\,h+i^{\,l}k,\,h+i^{\,l}k},
\]
where $i=\sqrt{-1}$. Consequently, the definition of a CP instrument may equivalently be formulated by replacing condition~(2), stated in terms of the complex measures $\cli_{a,h,k}$, with the requirement that the maps $\cli_{h,k}$ are $\cla^*$-valued measures.
\end{rmrk}

We next introduce a norm on the space of quantum instruments, analogous to the classical total variation norm for vector measures. The induced metric will be referred to as the \emph{total variation metric}.

\begin{dfn}[Total variation norm]
Let $\cli:\ox\to CP(\cla,\bh)$ be an instrument. For each $h,k\in\clh$, let $\cli_{h,k}$ denote the associated $\cla^*$-valued measure and let $|\cli_{h,k}|(X)$ denote its total variation over $X$. The \emph{total variation norm} of $\cli$ is defined by
\[
\|\cli\|_{\mathrm{TV}}
:=
\sup_{\|h\|=\|k\|=1}
|\cli_{h,k}|(X).
\]
The metric induced by $\|\cdot\|_{\mathrm{TV}}$ on the space of instruments will be called the \emph{total variation metric}.
\end{dfn}

\begin{thm}\label{completeness w.r.t TV metric}
    The set $Ins_{\clh}(X,\cla)$ of all instruments is complete with respect to the total variation metric.
\end{thm}
\begin{proof}
    It is straightforward to verify that $\|\cdot\|_{\mathrm{TV}}$ defines a norm. We prove that
    $Ins_{\clh}(X,\cla)$ is complete with respect to this norm.

Let $\{\cli^n\}_{n=1}^\infty$ be a Cauchy sequence in
$Ins_{\clh}(X,\cla)$. Then, for every $h,k\in\clh$ with
$\|h\|=\|k\|=1$,$
|(\cli^m-\cli^n)_{h,k}|(X)\longrightarrow0
\qquad (m,n\to\infty).$
Hence, for every $E\in\ox$,
\[
\|{\cli^m}_{h,k}(E)-{\cli^n}_{h,k}(E)\|
\le
|(\cli^m-\cli^n)_{h,k}|(E)
\le
|(\cli^m-\cli^n)_{h,k}|(X),
\]
so $\{{\cli^n}_{h,k}(E)\}$ is a Cauchy sequence in $\cla^*$.
Since $\cla^*$ is complete, there exists
$\cli_{h,k}(E)\in\cla^*$ such that
$
\cli_{h,k}(E)
=
\lim_{n\to\infty}\cli^n_{h,k}(E).$

Define $\cli(E):\cla\to\bh$ by
\[
\langle h,\cli(E)(a)k\rangle
=
\cli_{h,k}(E)(a),
\qquad
a\in\cla,\ h,k\in\clh.
\]
Following Remark \ref{reformulation of definition of instruments}, we next show that $\cli_{h,k}$ is a $\cla^*$-valued measure for $h,k\in\clh$.
Let $A=\cup_{j=1}^\infty A_j,$ union of disjoint sets $\{A_j\}$.
Fix $\epsilon>0$ and choose $N$ such that for $n,m\ge N$
\[
\|\cli^n-\cli^m\|_{\mathrm{TV}}<\epsilon/3. 
\]
Therefore,
 \begin{equation*}
    \Bigl\|\cli_{h,k}(A) - \sum_{j=1}^K \cli_{h,k}(A_j)\Bigr\| \leq \|\cli_{h,k}(A) - \cli_{h,k}^N(A)\| + \Bigl\|\cli_{h,k}^N(A) - \sum_{j=1}^K \cli_{h,k}^N(A_j)\Bigr\| + \sum_{j=1}^K \|\cli_{h,k}^N(A_j) - \cli_{h,k}(A_j)\|.
    \end{equation*}
Taking $K \to \infty$ and $N$ sufficiently large, each term is bounded by $\epsilon/3$, proving countable additivity.

Finally, for every finite collection
$\{a_i\}_{i=1}^r\subseteq\cla$ and
$\{h_i\}_{i=1}^r\subseteq\clh$,
\[
\begin{aligned}
\sum_{i,j}
\langle
h_i,
\cli(E)(a_i^*a_j)
h_j
\rangle
&=
\sum_{i,j}
\cli_{h_i,h_j}(E)(a_i^*a_j)  \\
&=
\lim_{n\to\infty}
\sum_{i,j}
\cli^n_{h_i,h_j}(E)(a_i^*a_j)
\ge0,
\end{aligned}
\]
since each $\cli^n(E)$ is completely positive.
Thus $\cli(E)\in CP(\cla,\bh)$ for every $E\in\ox$, so
$\cli$ is an instrument.
Finally, for every $h,k$,
\[
|(\cli^n-\cli)_{h,k}|(X)
=
\lim_{m\to\infty}
|(\cli^n-\cli^m)_{h,k}|(X),
\]
and taking the supremum over unit vectors gives
\[
\|\cli^n-\cli\|_{\mathrm{TV}}
\longrightarrow0.
\]
Hence $Ins_{\clh}(X,\cla)$ is complete.
\end{proof}

The foundation of integration theory rests upon a suitable notion of measurability and in particular, integrability for $\cla$-valued functions with respect to a given instrument $\cli$. In general, three notions of measurability are commonly considered: strong, Borel, and weak measurability. In this context, our focus will be on the first and third, namely, strong and weak measurability. We begin by introducing their definitions, with particular attention to strong measurability, as it forms the foundation of our subsequent construction of integration. Before doing so, we introduce the notion of a simple function, which, as in the classical theory, serves as the basic building block for the development of integration. Throughout this section, by $\cli,$ we denote a CP instrument $\cli:\ox\to CP(\cla,\bh).$
\begin{dfn}[Simple functions]\label{simple functions}
	A function $f:X\to \cla$ is said to be simple if there exists $a_1,a_2,\cdots,a_n\in\cla$ and $A_1,A_2,\cdots,A_n\in\ox$ such that $f=\sum_{i=1}^n a_i1_{A_i},$ where $1_{A_i}(x)=1$ if $x\in A_i$ and $1_{A_i}(x)=0$ if $x\notin A_i.$ 
	
\end{dfn}

	In his setting, Bartle employed the concept of semi-variation of a vector-valued measure to define various modes of convergence, including almost everywhere convergence, convergence in measure, and almost uniform convergence. In a similar spirit, we extend these notions to the vector-valued measures $\cli_{h,k}$ induced by an instrument $\cli$, for each $h, k \in \clh$.
\begin{dfn}[Almost everywhere convergence]
	Let $\{f_n:X\to\cla\}$ be a sequence of functions. We say that, $\{f_n\}$ converges to $f$ point-wise $\cli$-almost everywhere if $lim_n ~\|f_n(x)-f(x)\|=0,$ for $\|\cli_{h,k}\|$-almost every $x\in X$ and for all $h,k\in\clh.$
\end{dfn}

\begin{dfn}[Strongly $\cli$-measurable functions]
	A function $f:X\to \cla$ is called strongly $\cli$-measurable if for every $h,k\in \clh$ there exists a sequence of simple functions ${f^{h,k}}_n:X\to\cla,$ such that $\{{f^{h,k}}_n\}$ converges to $f,$  $\|\cli_{h,k}\|$-almost everywhere.  	
\end{dfn}

Next we introduce the weakly measurable functions.
\begin{dfn}[Weakly $\cli$-measurable functions]
	A function $f:X\to \cla$ is said to be weakly $\cli$-measurable if $a^*f$ is strongly $\cli$-measurable for every $a^*\in\cla^*.$
	
\end{dfn}
In general, the class of strongly measurable functions is a proper subclass of the class of weakly measurable functions. For instance, the identity map on $l^\infty((0,1))$ is weakly measurable, since every continuous linear functional on $l^\infty((0,1))$ is measurable when composed with the identity. However, it is not strongly measurable, as its range is non-separable; this follows from the following theorem (see Theorem~\ref{petti-type theorem}). 
However, under certain natural conditions, Petti (see Theorem II.2,\cite{uhl_vectormeasures}) showed that weak measurability coincides with strong measurability in the context of vector-valued measures. Here we establish a corresponding result for instruments. 
\begin{thm}\label{petti-type theorem}
	 A function $f:X\to \cla$ is strongly $\cli$-measurable if and only if
	\begin{enumerate}
		\item $f$ is $\cli$-essentially separably valued i.e. for every $h,k\in\clh$ there exists $A_{h,k}\in\ox$ with $\|\cli_{h,k}\|(A_{h,k})=0$ and $f(X\setminus A_{h,k})$ is a separable subset of $\cla$ \label{separable range} and
		\item f is weakly $\cli$-measurable.\label{strong implies weak}
	\end{enumerate}
\end{thm}
\begin{proof}
	
	Assume that $f$ is strongly $\cli$-measurable. Then for every $h,k\in\clh$ there exists a sequence of simple functions ${g_n}$ and a measurable set $A_{h,k}\in\ox$ such that $\lim_{n \to \infty}\| g_n(x) - f(x)\|=0,$ for all $x \in X \setminus A_{h,k}$.
	Each $g_n$ takes values in a finite-dimensional subspace $\cla_n := \textit{span}~ g_n(X)$. Define $$\cla'=\overline{\textit{span}\bigcup\cla_n}.$$ Then $\cla'$ is a separable subspace of $\cla$, and clearly $f(X \setminus A_{h,k}) \subset \cla'$.
	Hence, the assertion of \ref{separable range} follows.
	The proof of \ref{strong implies weak} follows immediately from the assumptions, thereby completing the proof of one direction.
	
	To prove the converse, for $h,k\in\clh,$ let $A_{h,k} \in\ox$ be chosen such that $\|\cli_{h,k}\|(A_{h,k})=0$ and $f(X\setminus A_{h,k})$ is separable. Let $\{a_n\}$ be a countable dense subset of $f(X\setminus A_{h,k}).$ 
	Let us define $g_n:\cla\to\{a_1,a_2,\cdots,a_n\}$ such that $\forall a\in \cla$, $g_n(a)$ is the one among $\{a_1,a_2,\cdots,a_n\}$ having the smallest index in $\{1,2,\cdots,n\}$ for which $$\|a-g_n(a)\|=min_{1\leq j\leq n}\|a-a_j\|.$$  Then $lim_n g_n(a)=a$ for all $a\in\cla.$ We define $f^{h,k}_n:X\setminus A_{h,k}\to\cla$ by $f^{h,k}_n(x)=g_n(f(x))$ for all $x\in X\setminus A_{h,k}.$ Clearly, $lim_n f^{h,k}_n(x)=f(x),~\forall~x\in X\setminus A_{h,k}$ i.e. $f^{h,k}_n$ converges to $f$ point-wise. It remains to verify that $f^{h,k}_n$ are measurable i.e. $({f^{h,k}_n})^{-1}(a_k)\in\ox$ for $1\leq k \leq n.$ Here $({f^{h,k}_n})^{-1}(a_k)=\{x\in X\setminus A_{h,k}: g_n(f(x))=a_k\}=\{x\in X\setminus A_{h,k}: \|f(x)-a_k\|=min_{1\leq j\leq n}\|f(x)-a_j\|\}.$ By Hahn Banach theorem, we choose a sequence $\{a^*_n\}\subset\cla^*$ such that $sup_n|a^*_n(a)|=\|a\|.$ Since infimum and supremum of countably many measurable functions are measurable, the function $$x\mapsto \|f(x)-a_j\|=min_{1\leq j\leq n}sup_m|a^*_m(f(x)-a_j)|$$ is measurable as well. Therefore, each set $({f^{h,k}_n})^{-1}(a_k)$ is measurable.
	
\end{proof}
\begin{crlre}\label{uniform limit of simple functions}
	It is evident from the previous theorem  for separably valued functions, both notions of measurability coincide. 
\end{crlre}
\begin{crlre}
	It is easy to verify from the Theorem \ref{petti-type theorem} that the collection of strongly-measurable functions on $X$ is a linear space which is closed under the operation of convergence in point-wise limit of sequences. 
\end{crlre}

In his formulation, Bartle defined the integral initially for all functions that can be approximated in measure by simple functions. He then showed that, under certain regularity conditions on the underlying vector measure, this class coincides with the set of functions that can be approximated by simple functions almost everywhere.
He referred to the corresponding regularity condition as the $*$-property of the underlying vector measure $\mu$, which is stated as follows: the measure $\mu$ has the $*$-property if there exists a finite non-negative measure $\nu,$ such that $$\mu(A)=0 ~\textit{if and only if}~ \nu(A)=0.\qquad(\textit{$*$-property})$$ 
This property ensures that the measure $\mu$ induces a well-behaved integration theory analogous to the scalar case. Bartle provided several sufficient conditions under which a vector-valued measure $\mu$ satisfies the $*$-property, bounded variation being one of them. Thus, whenever $\mu$ possesses bounded variation, the theory simplifies. Moreover, since for any instrument $\cli:\ox \to CP(\cla,\bh)$, each $\cla^*$-valued measure $\cli_{h,k}$ has bounded variation (as noted in Remark~\ref{vector measures associated to instruments}) for every $h,k \in \clh$, we can develop a consistent theory of integration with respect to an instrument by appealing to Bartle’s framework of vector-valued integration.

In the following, we recall the framework of Bartle integration adapted to our setting. 
Let $\mu:\ox \to \clb^*$ be a vector-valued measure. 
For a simple function $f:X\to\clb$ defined by $f(x) = \sum_i a_i 1_{A_i}(x)$, its Bartle integral with respect to $\mu$ is given by
\[
\int_X f\, d\mu := \sum_i \mu(A_i)(a_i).
\]

\begin{dfn}[Bartle integral]\label{bartle_integral}\index{Bartle integral}      
	Let $\mu:\ox \to \clb^*$ be a vector-valued measure. 
	A function $f:X\to\clb$ is said to be \emph{Bartle integrable} with respect to $\mu$ if there exists a sequence of simple functions $\{f_n\}$ on $X$ such that:
	\begin{enumerate}[label=(\roman*)]
		\item $f_n(x) \to f(x)$ pointwise $\|\mu\|$-almost everywhere; and
		\item 
        the sequence of indefinite integrals
\[
E\longmapsto\int_Ef_n\,d\mu
\]
converges uniformly in $E\in\ox$.
	\end{enumerate}
	In this case, we define
	\[
	\int_X f\, d\mu := \lim_{n \to \infty} \int_X f_n\, d\mu.
	\]
\end{dfn}

Next, we recall one of the most crucial results from Bartle’s paper ( \cite[Theorem 10]{bartle_integral}) in the present setting, which will be used on several occasions throughout this chapter.
\begin{thm}\label{vitali-theorem}
	 Let $\mu:\ox\to\clb^*$ be a vector-valued measure possessing the $*$-property.
	Suppose ${f_n}$ is a sequence of $\mu$-integrable functions such that:
	\begin{enumerate}[label=(\roman*)]
		\item the sequence $\{f_n\}$ converges to $f,$ $\|\mu\|$ almost everywhere;
		\item given $\epsilon>0$ there is a $\delta>0$ such that if $E\in\ox$ and $\|\mu\|(E)<\delta,$ then $|\int_E f_nd\mu|<\epsilon,$ for all $n\in\bbn .$ 
	\end{enumerate}
	Then $f$ is integrable with respect to $\mu$, and moreover,
	$\int_E fd\mu=lim_{n\to\infty}\int_E f_nd\mu.$ 
\end{thm}

As outlined earlier, to develop a rigorous theory of integration, we begin by defining the integral on the class of simple functions and subsequently extend it to the space of strongly measurable functions.
Let $f:X\to \cla$ be a simple function of the form $f=\sum_{i=1}^n a_i1_{A_i},$ as mentioned in Definition \ref{simple functions}, then we define its integral over any set $A\in\ox$ as $$\int_{A} fd\cli=\sum_{i=1}^n\cli(A\cap A_i,a_i).$$

\begin{dfn}[$\cli$-integrability]
	A function $f$ on $X$ is said to be \emph{$\cli$-integrable} if it satisfies the following conditions:
	\begin{enumerate}[label=(\roman*)]
		\item $f$ is strongly $\cli$-measurable;
		\item for every $h,k \in \clh$, the function $f$ is $\cli_{h,k}$-integrable in the sense of Bartle (Definition~\ref{bartle_integral}).
	\end{enumerate}
	In this case, the \emph{integral of $f$ with respect to $\cli$} is defined by
	\[
	\langle h, \Bigl(\int_X f\, d\cli\Bigr) k\rangle 
	= \int_X f\, d\cli_{h,k},
	\qquad \text{for all } h,k \in \clh.
	\]
\end{dfn}

\begin{rmrk}\label{relation to complex integration}
    It follows from the definition of integration, if $a\in\cla $ is fixed and $f: X\to\bbc$ is a measurable function, then for the function $af: X\to\cla,$  $\langle h,(\int af d\cli)k\rangle=\int f d\cli_{a,h,k}$ where $\cli_{a,h,k}$ is the complex measure mentioned in \ref{eq:notation for cli_a,h,k}.
\end{rmrk}

It follows from the definition of integrability and Theorem \ref{petti-type theorem} that :

\begin{thm}\label{separably valued bounded measurable functions integrability}
	Every $\cli$-essentially separable and measurable function $f : X \to \cla$ is $\cli$-integrable and if it is bounded then satisfies
	\[
	| \int_E f\, d\cli_{h,k}| \leq \|f\|_{\infty}\|\cli_{h,k}\|(E),
	\]
	for every $h,k \in \clh$ and $E \in \ox$.
\end{thm}

As a direct application of the previous theorem, we have the following corollary.
\begin{crlre}\label{cor : I-integrability of continuous functions}
	Let $X$ be a compact, Hausdorff space with $\ox$ be the Borel $\sigma$-algebra of $X$. Let $\cli:\ox\to CP(\cla,\bh)$ be a CP instrument then every $\cla$-valued continuous function on $X$ is $\cli$-integrable.
\end{crlre}
In the following subsections, we present several important applications of the integration theory developed above to the classical and $C^*$-convex structure of CP instruments.

\subsection{Application 1: Lyapunov-type theorem}\label{Application 1: Lyapunov-type theorem}

In this subsection, we establish a Lyapunov-type theorem for quantum instruments using the integration theory developed in the previous section. Lyapunov's classical theorem \cite{liapounoff1940fonctions} asserts that if $\{\mu_1,\ldots,\mu_n\}$ are finite positive non-atomic measures, then the set
\[
\{(\mu_1(A),\ldots,\mu_n(A)):A\in\ox\}\subseteq\mathbb{R}^n
\]
is compact and convex. This celebrated theorem has inspired extensive developments in the theory of vector measures and has found numerous applications in convex analysis, optimization, and related areas.

Plosker and Ramsey \cite{Plosker_Ramsey} revisited this problem in the setting of positive operator-valued measures (POVMs). Their proof establishes the convexity of the range for non-atomic POVMs under an additional structural assumption on the POVM. In a subsequent remark, however, they exhibited an explicit example that does not satisfy this assumption and asked whether the convexity of the range nevertheless persists. On the other hand, an important phenomenon underlying Lyapunov-type theorems, both in the classical and operator-valued settings, is the non-injectivity of the associated integration map \ref{integration representation}, a point that will become apparent in our approach.

We characterize those quantum instruments for which the Lyapunov property holds. Our proof requires no assumptions beyond the instrument's non-atomicity, thereby showing that the additional hypothesis imposed in the POVM setting is unnecessary and justifying Example 3.4 in \cite{Plosker_Ramsey}. 
To begin with, we fix the notation for the range of an instrument, which will play a central role in the results that follow.

\begin{dfn}
Let $\cli:\ox\to CP(\cla,\bh)$ be an instrument. The \emph{range} of $\cli$, denoted by $\mathcal{R}_\cli$, is the set
\[
\mathcal{R}_\cli=\{\cli(A):A\in\ox\}\subseteq CP(\cla,\bh).
\]
\end{dfn}

 Throughout, we shall work exclusively with non-atomic instruments. Indeed, for an atomic instrument—for example, one defined on a finite measurable space—the range $\mathcal{R}_\cli$ consists of finitely many completely positive maps and is therefore a finite discrete subset of $CP(\cla,\bh)$. Consequently, one cannot, in general, expect a Lyapunov-type convexity theorem to hold in the atomic setting.

Unless stated otherwise, the Hilbert space $\clh$ is assumed to be separable. In view of Remark~\ref{finite measure realization for faithful states}, we fix a faithful density operator $\rho\in\bh$ and denote the associated finite scalar measure $|\cli_\rho|$ by $\nu_\cli$. Throughout, let
$\mathcal{C}:=\{1_A:A\in\ox\}$
denote the collection of characteristic functions of measurable subsets of $X$, and write
$
\operatorname{co}(\mathcal{C})$
for the convex hull of $\mathcal{C}$.

We next recall the notions of atomic and non-atomic instruments introduced in
\cite{bhat_chongdar_sruthymurali}. Let
$\cli:\ox\to CP(\cla,\bh)$
be an instrument. A measurable set $A\in\ox$ is called an
\emph{atom} of $\cli$ if $\cli(A)\neq0$ and, for every measurable subset
$B\subseteq A$, either
$
\cli(B)=0
~~\text{or}~~
\cli(B)=\cli(A).$
An instrument is said to be \emph{atomic} if every measurable set
$A\in\ox$ satisfying $\cli(A)\neq0$ contains an atom, whereas it is
called \emph{non-atomic} if it admits no atoms. We emphasize that
\emph{non-atomic} is stronger than merely being \emph{non-atomic} in the
sense of ``not atomic''; the latter only means that the instrument fails
to be atomic, while the former requires the complete absence of atoms.

The atomic structure of an instrument is completely determined by its
associated POVM. More precisely, an instrument is atomic if and only if
its associated POVM is atomic. Consequently, every instrument admits a
unique decomposition as the sum of an atomic instrument and a non-atomic
instrument \cite{bhat_chongdar_sruthymurali}.

A basic example of a non-atomic instrument is obtained as follows.
Let $X=[0,1]$ be equipped with its Borel $\sigma$-algebra $\ox$, let
$p:\ox\to[0,1]$ denote the Lebesgue probability measure, and let
$\xi:\cla\to\mathbb C$ be a state. Define
$\widetilde{\xi}:\cla\to\bh$ by
$
\widetilde{\xi}(a)=\xi(a)I_{\clh},
~ a\in\cla.$
Then
$
\cli(A)=p(A)\widetilde{\xi},
~ A\in\ox,$
defines a non-atomic instrument.

Before going into the main results let us look for some preparotory results.

\begin{thm}\label{integration representation}
Let $\cli:\ox\to CP(\cla,\bh)$ be an instrument, and let $\nu_\cli$ be the finite measure associated with $\cli$ as mentioned at the beginning of this subsection. Then the map
\[
\psi_\cli:L^\infty(\nu_\cli)\longrightarrow \clb(\cla,\bh),
\qquad
\psi_\cli(f)(a)=\int_X a\,f\,d\cli,
\]
is a well-defined linear map which is weak$^*$--BW continuous. Furthermore, if
$0\leq f\leq1,$
then $\psi_\cli(f)$ is completely positive.
\end{thm}

\begin{proof}
Let $f\in L^\infty(\nu_\cli)$ and $a\in\cla$. Since the function
\[
a\cdot f:X\longrightarrow\cla,\qquad
x\longmapsto f(x)a,
\]
is bounded, measurable and separably valued, Theorem~\ref{separably valued bounded measurable functions integrability} ensures that $
\int_X a\,f\,d\cli
$
is well defined. Thus $\psi_\cli(f)$ is a well-defined map from $\cla$ into $\bh$. The linearity of $\psi_\cli$ follows immediately from the linearity of the integral.

To prove weak$^*$--BW continuity, let $\{f_\lambda\}$ be a net in
$L^\infty(\nu_\cli)$ converging to $f$ in the weak$^*$ topology. We must show that
\[
\psi_\cli(f_\lambda)\longrightarrow\psi_\cli(f)
\]
in the BW topology of $\clb(\cla,\bh)$. Fix $a\in\cla$ and
$h,k\in\clh$. By Remark~\ref{relation to complex integration},
$
\langle k,\psi_\cli(f_\lambda)(a)h\rangle
=
\int_X f_\lambda\,d\cli_{a,h,k},
$
and similarly,
$
\langle k,\psi_\cli(f)(a)h\rangle
=
\int_X f\,d\cli_{a,h,k}.
$
Since $\cli_{a,h,k}\ll\nu_\cli$, the measure
$\cli_{a,h,k}$ admits a Radon--Nikodým derivative with respect to
$\nu_\cli$. Therefore,
\[
\int_X f_\lambda\,d\cli_{a,h,k}
\longrightarrow
\int_X f\,d\cli_{a,h,k},
\]
by the definition of weak$^*$ convergence on $L^\infty(\nu_\cli)$. Hence,
\[
\langle k,\psi_\cli(f_\lambda)(a)h\rangle
\longrightarrow
\langle k,\psi_\cli(f)(a)h\rangle,
\]
which proves that $\psi_\cli$ is weak$^*$--BW continuous.

Finally, suppose that $0\leq f\leq1$. Choose a sequence of positive simple functions
$\{g_n\}$ such that
$
g_n\longrightarrow f$
uniformly on $X$. Since
\[
\psi_\cli(g_n)
=
\sum_{i=1}^{m_n}\alpha_i^{(n)}
\cli(E_i^{(n)},\cdot),
\]
where $\alpha_i^{(n)}\ge0$, each $\psi_\cli(g_n)$ is completely positive, being a finite positive linear combination of completely positive maps.

Now let $\{a_1,\dots,a_m\}\subseteq\cla$ and
$\{h_1,\dots,h_m\}\subseteq\clh$. Then
\[
\sum_{i,j=1}^m
\left\langle
h_i,
\psi_\cli(f)(a_i^*a_j)
h_j
\right\rangle
=
\lim_{n\to\infty}
\sum_{i,j=1}^m
\left\langle
h_i,
\psi_\cli(g_n)(a_i^*a_j)
h_j
\right\rangle.
\]
Each term on the right-hand side is non-negative since
$\psi_\cli(g_n)$ is completely positive. Passing to the limit yields
\[
\sum_{i,j=1}^m
\left\langle
h_i,
\psi_\cli(f)(a_i^*a_j)
h_j
\right\rangle
\ge0,
\]
proving that $\psi_\cli(f)$ is completely positive.
\end{proof}

\begin{crlre}\label{compactness of the range}
Since the closed unit ball of $L^\infty(\nu_\cli)$ is weak$^*$-compact by the Banach--Alaoglu theorem and $\psi_\cli$ is weak$^*$--BW continuous, it follows that the image of the unit ball under $\psi_\cli$ is BW-compact in $\clb(\cla,\bh)$.
\end{crlre}
\begin{rmrk}
In finite dimensions, the map $\psi_\cli$ is automatically non-injective for every non-atomic instrument. Consequently, every non-atomic finite-dimensional instrument has a BW-compact convex range as we will see in the following Theorem \ref{rangecharacterization}.
This phenomenon, however, fails in infinite dimensions, as the following standard example illustrates (\cite{uhl_vectormeasures}).
\end{rmrk}

\begin{xmpl}\label{injective example}
Since every POVM is, in particular, an instrument, it suffices to consider a POVM. Let
$
\nu:\mathcal{O}([0,1])\longrightarrow\mathcal{B}(L^2([0,1]))
$
be the projection-valued measure defined by
$
\nu(A)=M_{1_A},
$
where $M_{1_A}$ denotes the multiplication operator by the characteristic function $1_A$, and $\mathcal{O}([0,1])$ denotes the Borel $\sigma$-algebra on $[0,1]$. Since the Lebesgue measure is non-atomic, it follows immediately that $\nu$ is a non-atomic POVM.
If $\rho$ is any faithful density operator on $L^2([0,1])$, then the associated scalar measure
$
\nu_\rho(A)=\operatorname{Tr}(\rho\,M_{1_A})
$
is equivalent to the Lebesgue measure.
Consider the integration map
\[
\psi_\nu:L^\infty([0,1],m)\longrightarrow\mathcal{B}(L^2([0,1])),
\qquad
f\longmapsto\int f\,d\nu.
\]
Note that,
\[
\int f\,d\nu=M_f,
\]
where $M_f$ denotes the multiplication operator induced by $f$. Consequently, if $\psi_\nu(f)=0$, then $M_f=0$, which implies that $f=0$ almost everywhere with respect to the Lebesgue measure. Hence $\psi_\nu$ is injective.

\end{xmpl}

 We are now ready to prove the main theorem of this subsection.

\begin{thm}[Lyapunov-type theorem]\label{rangecharacterization}
Let $\cli:\ox\to CP(\cla,\bh)$ be a non-atomic instrument. Then the following are equivalent.

\begin{enumerate}
    \item The range $\mathcal{R}_\cli$ is convex.
    \item For every measurable set $E\in\ox$ with $\nu_\cli(E)\neq0$, the restricted integration map
    \[
    \psi_\cli^E:L^\infty(E,\nu_\cli|_E)\longrightarrow\clb(\cla,\bh),
    \qquad
    f\longmapsto\int_E f\,d\cli,
    \]
    is not injective.
\end{enumerate}
\end{thm}

\begin{proof}
We first prove $(2)\Rightarrow(1)$.

Let $(\clb_+)_1=\{f\in L^\infty(X,\nu_\cli):0\leq f\leq1\}.$  Therefore, it suffices to prove that $
\mathcal{R}_\cli=\psi_\cli((\clb_+)_1).$

The inclusion $
\mathcal{R}_\cli\subseteq\psi_\cli((\clb_+)_1)$
is immediate, since $
\psi_\cli(1_A)=\cli(A),
~ A\in\ox.$
For the converse inclusion, fix
$
\phi\in\psi_\cli((\clb_+)_1),
$
and consider the fiber
$
F_\phi
=
\{f\in(\clb_+)_1:\psi_\cli(f)=\phi\}.$
Since $\psi_\cli$ is weak$^*$--BW continuous (Theorem \ref{integration representation}), $F_\phi$ is a weak$^*$-closed subset of the weak$^*$-compact convex set $(\clb_+)_1$. Hence $F_\phi$ is weak$^*$-compact and convex. By the Krein--Milman theorem, $F_\phi$ admits an extreme point, say $g$.

We claim that $g$ must be the characteristic function of a measurable subset of $X$. Suppose otherwise. Then there exist $\alpha\in(0,1)$ and a measurable set $A_\alpha\in\ox$ with
$
\nu_\cli(A_\alpha)\neq0
$
such that
$
\alpha 1_{A_\alpha}
\leq
g1_{A_\alpha}
\leq
(1-\alpha)1_{A_\alpha}.
$

By assumption, the restricted integration map
\[
\psi_\cli^{A_\alpha}:L^\infty(A_\alpha,\nu_\cli|_{A_\alpha})
\longrightarrow
\clb(\cla,\bh)
\]
is not injective. Therefore, there exists a non-zero function
$
h\in L^\infty(A_\alpha,\nu_\cli|_{A_\alpha})$
such that
$
\psi_\cli^{A_\alpha}(h)=0
$ and $
\|h\|_\infty\leq1.
$

Define, $g_\pm=g\pm\alpha h.$
Since $\|h\|_\infty\leq1$ and 
$\alpha\leq g\leq1-\alpha$
on $A_\alpha$, we obtain $
0\leq g_\pm\leq1.$
Moreover,
\[
\psi_\cli(g_\pm)
=
\psi_\cli(g)\pm\alpha\psi_\cli^{A_\alpha}(h)
=
\phi,
\]
and $
g=\frac12(g_++g_-).$
Since $h\neq0$, we have $g_+\neq g_-$, contradicting the extremality of $g$. Hence every extreme point of $F_\phi$ is a characteristic function. Consequently,
$
\phi=\psi_\cli(1_A)=\cli(A)
$
for some measurable set $A\in\ox$. This proves that
$
\mathcal{R}_\cli=\psi_\cli((\clb_+)_1),
$
and therefore $\mathcal{R}_\cli$ is BW-compact and convex.

We now prove $(1)\Rightarrow(2)$.
Suppose that $\mathcal{R}_\cli$ is convex.

Since
\[
\mathcal{R}_\cli
=
\psi_\cli(\mathcal{C})
=
\psi_\cli(\operatorname{co}(\mathcal{C})),
\]
and
\[
\mathcal{C}
\subsetneq
\operatorname{co}(\mathcal{C}),
\]
the map $\psi_\cli$ is not injective on $\operatorname{co}(\mathcal{C})$. The same argument applies to every measurable subset $E\in\ox$, proving that
\[
\psi_\cli^E:\operatorname{co}(\mathcal{C}_E)\longrightarrow CP(\cla,\bh)
\]
is not injective. This completes the proof.
\end{proof}

 \begin{crlre}\label{range of non-atomic POVM}
    Let $\eta:\ox\to\bh$ be a non-atomic POVM. Then if the integration map $\psi_\eta$ associated to $\eta$ is non-injective then the range of $\eta$ is convex. 
\end{crlre}

\subsection{Application 2: Correspondence between instruments and CP maps}\label{Application 2: Correspondence between instruments and CP maps}
We now move our attention to the case when $X$ is a topological space and $\clo(X)$ is to be the Borel $\sigma$-algebra on $X$. In this setting, to establish the correspondence between regular instruments and completely positive maps, we first recall the notion of regular vector-valued measures, which naturally generalizes the classical concept of regularity from scalar to operator-valued contexts. Here regularity is crucial to establish the one-one correspondence.
\begin{dfn}[Regular vector measures]
	Let $\clb$ be a Banach space, $X$ be a topological space and $\clo(X)$ be the $\sigma$-algebra of Borel subsets of $X$. A vector measure $\mu$ on $(X,\clo(X))$ with values in $\clb$ is said to be regular if for every Borel set $E\in\clo(X)$ and $\epsilon>0$, there exists an open set $O$ containing $E$ and a compact set $C$ contained in $E$ such that $||\mu(O \setminus C)||<\epsilon$.
\end{dfn}
\begin{rmrk}\label{the set of all regular measures forms a Banach space}
	Recall that, for a compact Hausdorff space $X$, the space $M(X,\clb)$ of all regular $\clb$-valued measures of bounded variation is a Banach space under the total variation norm
\[
\|\mu\|=|\mu|(X).
\]
Theorem~\ref{completeness w.r.t TV metric} shows that an analogous completeness phenomenon holds for the set of completely positive instruments equipped with the total variation metric.
\end{rmrk}	 


Next we introduce regularity for instruments in a similar manner as discussed in \cite{manishcmp} in the context of POVMs.

\begin{dfn}[Regular instruments]\label{Regular instruments}
		Let $X$ be a topological space and $\ox$ be the Borel $\sigma$-algebra of $X$. An instrument $\cli: \ox\to CP(\cla,\bh)$ is said to be a regular instrument if $\cli_{h,k},$ as defined in \ref{measurecorrespondence}, is a regular $\cla^*$ valued measure on $(X,\ox)$ for all $h,k \in \clh$. 
	\end{dfn}

	\begin{notation}
	   We denote by $\mathcal{R}\text{-}Ins_{\mathcal{H}}(X,\cla)$ the collection of all regular CP instruments on $(X,\ox)$ with values in $CP(\cla,\bh)$.  
Similarly, $\mathcal{R}\text{-}I_{\mathcal{H}}(X,\cla)$ denotes the set of all regular UCP instruments on $(X,\ox)$.

	\end{notation}

In the following theorem, we extend the classical correspondence between regular positive operator-valued measures (POVMs) on compact Hausdorff spaces and completely positive (CP) maps on commutative $C^*$-algebras to the broader framework of CP instruments. This result naturally generalizes both the Riesz–Markov theorem (\cite[Theorem III.5.7]{conway_book}) and the CP–POVM correspondence ( \cite[Theorem 4.5]{paulsen_book}) to the instrument setting.

\begin{thm}\label{CP and instrument correspondence}
	Let $X$ be a compact Hausdorff space and $\cla$ be a unital $C^*$-algebra. Then there is a one-to-one correspondence between regular CP instruments $\cli$ on $(X,\ox)$ and CP maps from $C(X,\cla)$ to $\bh$, where $C(X,\cla)$ denotes the $C^*$-algebra of $\cla$-valued continuous functions on $X.$ 
\end{thm}
\begin{proof}
	
	Let $\cli: \ox \to CP(\cla,\bh)$ be a regular instrument. It follows from Corollary \ref{cor : I-integrability of continuous functions} that for $f \in C(X,\cla)$, the following map $\Phi_{\cli}: C(X,\cla)\to\bh$ defined by,
	$$\Phi_{\cli}=\int f~d\cli,$$
	is well defined.

	We claim that the association 
	$C(X,\cla)\ni f \to \Phi_\cli(f) \in \bh$ defines a CP map such that,
	$$\langle h,\Phi_\cli(f)k\rangle =\int_X f~d\cli_{h,k}\,.$$
	To prove the claim, for any $\{f_1,\cdots f_n\}\subset C(X,\cla)$ and $\{h_1,\cdots, h_n\}\subset \clh$, we need to verify that $\sum_{i,j} \langle h_i,\Phi_\cli(f_i^*f_j)h_j\rangle \geq0$. Observe that,
	\begin{eqnarray*}
		\sum_{i,j} \langle h_i,\Phi_\cli(f_i^*f_j)h_j\rangle &=&\sum _{i,j}\int f_i^*f_j~ d\cli _{h_i,h_j}
	\end{eqnarray*}
	Hence, to prove the claim, it is enough to prove that for any $\{f_1,\cdots f_n\}\subset C(X,\cla)$ and $\{h_1,\cdots, h_n\}\subset \clh$,  
	\begin{eqnarray*} 
		\sum _{i,j}\int f_i^*f_j~ d\cli _{h_i,h_j}\geq 0.
	\end{eqnarray*}
	First consider the indicator functions $g_i=a_i\cdot1_{A_i}\in B(X,\cla)$ for $a_i\in \cla$ and $A_i\in \clo(X)$ where $ B(X,\cla)$ denotes the collection of all bounded $\cla$ - valued measurable functions. Then we have,
	\begin{eqnarray}\nonumber
		\sum_ {i,j}\int g_i^*g_j~d\cli _{h_ih_j}&=&\sum_{i,j}\int a_i^*a_j1_{A_i\cap A_j}~d\cli _{h_ih_j}\\\nonumber
		&=& \sum _{i,j}\langle h_i,\cli(A_i\cap A_j,a_i^*a_j)h_j\rangle\label{correspondence1}\\
		&\geq &0,
	\end{eqnarray}
	where the last inequality follows from a calculation carried out in the proof of Theorem~2.7 of \cite{bhat_chongdar_sruthymurali}. Since $C(X,\cla)$ is the $\cst$-norm closed linear span of the functions $af,~a\in~\cla,~f\in~C(X).$ Hence any $f\in C(X,\cla)$ can be approximated uniformly by a sequence of simple functions in every regular $\cla^*$ measures $\cli_{h_i,h_j},~ h_i,h_j\in \mathcal{H}$. Let $\{g_{i,m}\}$ be a sequence of simple functions converging uniformly to $f_i$.
	Therefore,
	\begin{eqnarray*}
		\sum_{i,j}\int f_i^*f_j~ d\cli_{h_i,h_j} 
		&=& \sum _{i,j}\int (\lim_m g_{i,m}^* g_{j,m})~ d\cli_{h_i,h_j}~~~~\text{(by Theorem \ref{vitali-theorem} )}\\
		&=&	\sum_ {i,j}\lim_m \int (g_{i,m}^* g_{j,m})~ d\cli_{h_i,h_j}\\
		&=& \lim_m \sum _{i,j} \int g_{i,m}^*g_{j,m}~d\cli_{h_i,h_j}\\
		&=& \lim_m \sum _{i,j}\langle h_i, \Phi_\cli (g_{i,m}^*g_{j,m})h_j \rangle\\
		&\geq& 0\quad\quad (\text{By Equation \ref{correspondence1}}), 
	\end{eqnarray*}
	
	proving the claim.

	To establish the other way correspondence, note that the dual $C(X,\cla)^*$ can be identified with $M(X,\cla^*),$ (see \cite{ryan_tensorproduct_book}, page 112), by the following correspondence:
	\begin{eqnarray}\label{correspondence}\nonumber
		M(X,\cla^*)\ni\nu \longmapsto (f \to \int f~d\nu) \in C(X,\cla)^*.
	\end{eqnarray}
	
	Now let $\Phi: C(X,\cla)\to \clb(\clh)$ be a UCP map. For each $h,k \in \clh$, the map $C(X,\cla)\ni f \longmapsto \langle h,\Phi(f)k\rangle$ defines a bounded linear functional on $C(X,\cla)$. Hence it corresponds to a regular $\cla^*$ - valued measure $\nu_{h,k}\in M(X,\cla^*)$ on $(X,\clo(X))$ such that 
	\begin{eqnarray}\label{correspondence3}
		\int f ~d\nu_{h,k}= \langle h,\Phi(f)k\rangle.
	\end{eqnarray}
	
	It is clear from Equation \ref{correspondence3} that $\nu_{h,h}$ is a $\cla^*_+$ - valued measure for all $h \in \clh$. Now for $h,k \in \clh$, the total variation of the $\cla^*$ - valued measure $\nu_{h,k}$ is given by
	$$|\nu_{h,k}|(A)=\sup \{\sum_i\lVert\nu_{h,k}(E_i)\rVert: \{E_1,\cdots,E_n\} \text{ is a partition of $A$}\}.$$
	Since $\nu_{h,k}$ are of bounded variation, we have $|\nu_{h,k}|$ is a finite positive measure on $(X,\clo(X))$. Now for each bounded $\cla$-valued measurable function $g \in B(X,\cla)$, the map $(h,k) \to \int g~d\nu_{h,k}$ defines a bounded sesquilinear form $\clh \times \clh \to \bbc$. Hence by Riesz theorem, we get a unique bounded operator $\tilde\Phi(g) \in \clb(\clh)$ satisfying,
	\begin{eqnarray}\nonumber
		\langle h, \tilde\Phi(g)k \rangle=\int ~g~ d\nu_{h,k}.
	\end{eqnarray}  
	It can be checked that $\tilde{\Phi}(g)\geq 0$  if $g\geq 0$ in $B(X,\cla)$. For $A\in \clo(X), a \in \cla$, define $\cli_\Phi(A,a)=\tilde{\Phi}(a.1_A)$ such that,
	\begin{eqnarray}\label{correspondence2}
		\langle h,\cli_\Phi(A,a)k\rangle=\langle h,\tilde \Phi(a.1_A)k\rangle=\int ~a.1_A~d\nu_{h,k}=\nu_{h,k}(A)(a).
	\end{eqnarray}
	We claim that $\cli_\Phi: \ox \to CP(\cla,\bh)$ is an instrument. Let $a \in \cla$ be fixed. We need to show that $\cli_\Phi(\cdot,a) $ defines an operator valued measure and in particular if $a\in \cla_+$, we need to prove that $\cli_ \Phi(\cdot, a)$ is a POVM. Note that for any countable collection $\{B_i\}\subset \clo(X)$ of disjoint measurable subsets, we have the following:
	
	$$\langle h,\cli_\Phi(\cup_iB_i,a)k\rangle=\langle h,\tilde \Phi(a.1_{\cup_iB_i})k\rangle= \nu_{h,k}(\cup_iB_i)(a)=\sum_i \nu_{h,k}(B_i)(a)=\sum_i\langle h,\cli_\Phi(B_i,a)k\rangle $$
	for all $ h, k \in \clh$ proving that $\cli_\Phi(\cdot,a) $ defines an operator valued measure.  Now since for each $A \in \clo(X)$,  $\nu_{h,h}(A)$ is a positive linear functional on $\cla$, it is immediate from Equation \ref{correspondence2} that $\cli_ \Phi(\cdot, a)$ defines a POVM for all $a\in \cla_+$.
	
	To prove that $\cli_\Phi: \ox \to CP(\cla,\bh)$ is an instrument, we also need to show that $\cli_ \Phi(A,\cdot)$ defines a CP map for every $A \in \clo(X).$ Let $A \in \clo(X)$ be fixed. We need to see that,
	\begin{eqnarray}\label{correspondence5}
		\sum_{i,j} \langle h_i,\cli_\Phi(a_i^* a_j,A)h_j\rangle\geq 0\quad \text{for $\{a_1,\cdots,a_n\}\subset \cla$ and $\{h_1,\cdots,h_j\}\subset \mathcal{H}$}.
	\end{eqnarray}  Since $X$ is compact Hausdorff, we have $C(X)\subset L^\infty(X,\nu)\subset L^1(X,\nu)$ for any finite scalar measure $\nu$ on $X$. Now by the denseness of $C(X)$ in $L^1(X,\nu)$, there exists a sequence of positive continuous functions  $\{f_n\}\subset C(X)$ such that $f_n$ converges to $1_A$ in $L^1$ norm and so there exists a subsequence $f_{n_k}$ converges point-wise $\nu$ almost everywhere to $1_A$. In fact by Lusin's theorem we can take the sequence $\{f_{n_k}\}$ to be uniformly bounded. 
	
	To prove Equation \ref{correspondence5}~, let us fix $\{a_1,\cdots,a_n\}\subset \cla$ and $\{h_1,\cdots,h_j\}\subset \mathcal{H}$. Consider the finite positive measure $\nu=\sum_{i,j}|\nu_{h_i,h_j}|$. Then it is clear that $|\nu_{h_i,h_j}|\ll \nu$ for $i,j=1,\cdots,n$. We have from the above discussion that $f_{n_k}$ converges point-wise $\nu$ almost everywhere to $1_A$. In particular $f_{n_k}$ converges point-wise $|\nu_{h_i,h_j}|$ almost everywhere to $1_A$, for $i,j=1,\cdots,n$. Hence $a_i^*a_j\cdot f_{n_k}$ converges point-wise $\nu_{h_i,h_j}$ almost everywhere to $a_i^*a_j \cdot 1_A$ for $ i,j=1,\cdots,n$. Since $f_{n_k}$'s are uniformly bounded, we can conclude that $a_i^*a_j\cdot f_{n_k}$'s are uniformly bounded and hence by the Theorem \ref{vitali-theorem} we have,
	\begin{eqnarray}\label{correspondence4}
		\lim_k \int a_i^*a_j\cdot f_{n_k}d\nu _{h_i,h_j}=\int a_i^*a_j\cdot 1_A~d\nu _{h_i,h_j},\quad \text{for }  i,j=1,\cdots,n.
	\end{eqnarray}

	Now observe that,
	\begin{eqnarray*}
		\sum_{i,j} \langle h_i,\cli_\Phi(a_i^* a_j,A)h_j\rangle&=& \sum_{i,j}\langle h_i,\tilde{\Phi}((a_i^*a_j)\cdot1_A)h_j \rangle \\
		&=& \sum_{i,j}\int a_i^* a_j\cdot 1_A~ d\nu_{h_i,h_j}\\
		&=&	\sum_ {i,j}\lim_k \int a_i^* a_j\cdot f_{n_k}~ d\nu_{h_i,h_j}\quad\text{(By Equation \ref{correspondence4})}\\
		&=& \lim_k \sum _{i,j} \int a_i^*a_j\cdot f_{n_k}~d\nu_{h_i,h_j}\\
		&=& \lim_k \sum _{i,j}\langle h_i, \Phi ((a_i^*a_j)\cdot f_{n_k})h_j \rangle\\
		&=& \lim_k \sum _{i,j} \langle h_i ,\Phi ((\sqrt{f_{n_k}}a_i)^*(\sqrt{f_{n_k}}a_j)h_j \rangle\\
		&\geq& 0\quad(\text{since $\Phi$ is CP}), 
	\end{eqnarray*}
	finishing the proof of Equation \ref{correspondence5}.
	
	An easy verification shows that:\\
	(i) $\cli_{\Phi_{\cli}}=\cli$;\\
	(ii) $\Phi_{\cli_\Phi}=\Phi.$\\
	concluding the proof of the bijection between the sets $\mathcal{R}\text{-}Ins_{\mathcal{H}}(X,\cla)$ and $CP(C(X,\cla),\bh)$.

\end{proof}
\begin{rmrk}
    To prove the converse direction, one may appeal to dilation theory, thereby avoiding the technical intricacies of integration theory. We outline this alternative approach below. Let $\Phi: C(X,\cla)\to\bh$ be a completely positive map, and consider its minimal Stinespring dilation. Since, $C(X,\cla)\equiv C(X)\otimes\cla,$ $C(X)$ is nuclear, the tensor product is uniquely determined. Thus, we may write $\Phi(f\otimes a)=V^*\pi_{E}(f)\pi(a)V.$ By the CP–POVM correspondence (see \cite[Theorem 4.5]{paulsen_book}), there exists a spectral measure $E:\ox\to\bh$ on the measurable space $(X,\ox).$ Define the map $\cli_{\Phi}:\ox\times\cla\to\bh$ given by $\cli_\Phi(A,a)=V^*E(A) \pi(a)V,~a\in\cla,~A\in\ox.$ It is straightforward to check that $\cli_\Phi$ defines a CP instrument. Moreover, a routine argument shows that $\Phi=\Phi_{\cli_\Phi}.$ 
\end{rmrk}
\begin{rmrk}
    In the special case when our $(X,\ox)$ be countable discrete space, then the set of instruments corresponds to the set of CP maps $CP(\cla, L^\infty(X,\bh))$
\end{rmrk}
\subsection{BW Topology on Instruments}
Let $X$ be a topological space, and let $I_{\clh}(X,\cla)$ denote the set of all unital instruments from $\ox$ to $CP(\cla,\bh)$. In this context, we introduce a topology on $I_{\clh}(X,\cla)$, inspired by the work of Bhat and Kumar~\cite{manishcmp}, which extends the bounded weak (BW) topology previously developed for CP maps. 

Rather than working with the entire class of simple functions from the topological space  $X$ into the $C^*$-algebra $\cla,$ we restrict attention to a distinguished subclass, denoted by
$\cbxa.$ This space consists of all bounded, continuous, separably valued functions on $X$ taking values in $\cla$, with the additional requirement that for every $f \in \cbxa$, admits a sequence of simple functions converging uniformly to $f.$ 

This choice leads to a natural extension of the classical weak topology on the space of probability measures on a topological space, namely the topology induced by bounded continuous functions. Moreover, when $X$ s a locally compact Hausdorff space, the resulting topology coincides with the topology induced by compactly supported continuous functions.
Recall that, for any instrument $\cli$, the corresponding $\cla^*$-valued measure $\cli_{h,k}$ is defined as in Proposition~\ref{measurecorrespondence}.
We now introduce a topology on $I_{\clh}(X,\cla)$ by specifying the convergence of nets.
\begin{dfn}\label{definition of bw topology on instruments}\index{BW (bounded weak) topology! on instruments}
	Let ${\cli^j}$ be a net in $I_{\clh}(X,\cla)$ and $\cli \in I_{\clh}(X,\cla)$.
	We say that $\cli^j \to \cli$ in the {\em bounded weak (BW)\/} topology if
	\begin{align*}
		\int_Xfd\cli^j_{h,k}\to\int_X fd\cli_{h,k}
	\end{align*}
	for all $f\in \cbxa$ and $h,k\in\clh$.
\end{dfn}

It is worth noting that the topology defined above on $I_{\clh}(X,\cla)$ is the weakest topology that renders all maps of the form: $\cli\mapsto \int_X fd\cli_{h,k}$ from $I_{\clh}(X,\cla)$ to $\bbc$, continuous for all $f\in \cbxa$ and $h,k\in\clh$. Consequently, for any fixed $\cli\in I_{\clh}(X,\cla)$, a neighborhood basis at $\cli$ is given by the sets
\begin{align*}
	O=\left\{\clj\in I_{\clh}(X,\cla): \left|\int_X f_id\clj_{h_i,k_i}-\int_X f_id\cli_{h_i,k_i}\right|<\epsilon, 1\leq i\leq n\right\},
\end{align*}
where $f_i\in \cbxa$, $h_i,k_i\in\clh$ for $1\leq i\leq n$, $\epsilon
>0.$ 
The structure of this topology is analogous to the weak topology commonly employed in classical probability theory, thereby generalizing the notion to the setting of operator-valued instruments. Moreover, following the correspondence established in Theorem~\ref{CP and instrument correspondence}, it is evident that the bounded weak topology introduced above is naturally linked to the BW topology on the space of completely positive maps $CP(C(X,\cla),\bh),$ where $X$ is a compact, Hausdorff space. 
For a net $\cli^i$ and $\cli\in \mathcal{R}\text{-}Ins_{\mathcal{H}}(X,\cla)$, recall that the associated completely positive map is defined as \[ \Phi_\cli(f)=\int_Xfd\cli~ \textit{for all}~ f\in C(X,\cla), \] It then follows that the convergence $\cli^i\to\cli$ in $\mathcal{R}\text{-}I_{\mathcal{H}}(X,\cla)$ is equivalent to the convergence $\Phi_{\cli^i}(f)\to\Phi_\cli(f)$ in WOT for all $f\in C(X,\cla).$
The following proposition formalizes this observation and, in essence, states that the spaces $\mathcal{R}-I_\mathcal{H}(X,\cla)$ and $UCP(C(X,\cla),\bh)$ are topologically homeomorphic with respect to their respective bounded weak topologies.

\begin{ppsn}\label{BW topolgy are homeomorphic for regular instruments and UCP maps}
	Let $\cli^i$ be a net in $\mathcal{R}-I_\mathcal{H}(X,\cla)$ and $\cli\in\mathcal{R}-I_\mathcal{H}(X,\cla)$. Then the following are equivalent:
	\begin{enumerate}
		\item $\cli^i\to\cli$ in $\mathcal{R}-I_\mathcal{H}(X,\cla)$ (and, in $I_{\clh}(X,\cla)$) in BW topology.
		\item $\Phi_{\cli^i}\to\Phi_\cli$ in BW topology in $UCP(C(X,\cla),\bh)$.
	\end{enumerate}
\end{ppsn}

 To begin with, we formally define C*-convexity for the collection of all normalized instruments $I_\mathcal{H}(X,\cla).$
	
	\begin{dfn} [$C^*$-convexity]\label{cstarconvexity}
		For any $\cli_i\in I_\clh(X,\cla)$ and $T_i\in\bh$, $1\leq i\leq  n$ with $\sum_{i=1}^n T_i^*T_i=I_\mathcal{H}$, a sum of the form
		\begin{equation}\label{C^*sum}
			\cli(\cdot)=   \sum_{i=1}^nT_i^*\cli_i(\cdot)T_i
		\end{equation}
		is called a \emph{$\cst$-convex combination}\index{$C^*$-convex combination} for $\cli$. The operators $T_i$'s here
		are called {\em $\cst$-coefficients}\index{$C^*$-coefficients}. When $T_i$'s are invertible,
     		the sum in \eqref{C^*sum} is called a
		\emph{proper $\cst$-convex combination}\index{$C^*$-convex combination!proper-} for $\cli$.
	\end{dfn}
	
	Observe that $I_\clh(X,\cla)$ is a {\em $\cst$-convex set}\index{$C^*$-convex set} in the sense that it is closed
	under $\cst$-convex   combinations i.e. $\sum_{i=1}^n{T_i}^*\cli_i(\cdot)T_i\in I_\clh(X,\cla)$, whenever $\cli_i\in I_\clh(X,\cla)$ and $T_i\in\clb(\clh)$ satisfying $\sum_{i=1}^n{T_i}^*T_i=I_\clh$. 
	
	\begin{dfn}[$C^*$-extreme point]\label{cstarextreme}
		An instrument $\cli:\ox\to CP(\cla,\bh)$ is called a \emph{$C^*$-extreme point}\index{$C^*$-extreme point} in
		$I_\clh(X,\cla)$ if, whenever  $\sum_{i=1}^nT_i^*\cli_i(\cdot)T_i$
		is a proper $C^*$-convex combination of $\cli$, then each $\cli_i$ is unitarily equivalent to $\cli$ i.e. there are unitary operators $U_i\in\clb(\clh)$
		such that $\cli_i(\cdot)=U_i^*\cli(\cdot)U_i $ for $1\leq i\leq n.$
	\end{dfn}
	
\subsection{Application 3: Krein-Milman type theorem}\label{Application 3: Krein-Milman type theorem}
The Krein–Milman theorem stands as one of the fundamental results in classical functional analysis. It asserts that, in any locally convex topological vector space, every compact convex subset is the closure of the convex hull of its extreme points. Motivated by this classical result, it is natural to seek an analogue of the Krein–Milman theorem in the framework of $\cst$-convexity, particularly within the space of CP instruments. To this end, we first establish the following proposition, whose proof closely follows the classical line of argument.

\begin{ppsn}\label{prop:normalized instruments with finite support are dense}
Let $X$ be a topological space and let $\clh$ be a Hilbert space. We regard singleton sets as measurable elements of $\ox.$ 
	Then the collection of all normalized instruments on $(X,\ox)$ that are concentrated on finite subsets of $X$ is dense in $I_{\clh}(X,\cla).$
	\end{ppsn}
    \begin{proof}
	Let $\cli \in I_{\clh}(X,\cla)$, and let $E$ be a typical open neighbourhood of $\cli$ in $I_{\clh}(X,\cla)$ of the form
	\[
	E = \left\{ \clj \in I_{\clh}(X,\cla) : 
	\left|\int_X f_i\, d\clj_{h_i,k_i} - \int_X f_i\, d\cli_{h_i,k_i}\right| < \epsilon, \ 1 \leq i \leq n \right\},
	\]
	for some fixed $f_i \in \cbxa$, $h_i, k_i \in \clh$ $(1 \leq i \leq n)$, and $\epsilon > 0$.  
	We shall construct an element $\clj \in E$ that is concentrated on a finite subset of $X$, which will yield the desired density result.
	
	\smallskip
	For each $i \in \{1, \ldots, n\}$, there exists a simple function $g_i$ on $X$ such that
	\[
	\sup_{x \in X } \|f_i(x) - g_i(x)\| < \frac{\epsilon}{2M},
	\]
	where $M$ is a positive constant satisfying $M > \sup_i \|h_i\|\,\|k_i\|$. Since each $g_i$ is a simple function, there exists a finite partition $\{A_{ij}\}$ of $X$ and corresponding elements $\{a_{ij}\} \subseteq \cla$ (where $j$ ranges over some finite index set $\Lambda_i$, depending on $i$) such that
	\[
	g_i = \sum_{j \in \Lambda_i} a_{ij} 1_{A_{ij}}
	\]
	for each $i = 1, \ldots, n$. Now, for each $i$ and $j$, choose a point $x_{ij} \in A_{ij}$ and define
	\[
	\clj := \sum_{i=1}^n \sum_{j \in \Lambda_i} \delta_{x_{ij}}(\cdot)\, \cli(A_{ij})(\cdot).
	\]
	Clearly, $\clj$ is an instrument concentrated on the finite subset $\{x_{ij}\} \subseteq X$. Moreover,
	\[
	\clj(X)(1_\cla) = \sum_{i=1}^n \sum_{j \in \Lambda_i} \cli(A_{ij})(1_\cla)
	= \cli(X)(1_\cla) = I_\clh,
	\]
	hence $\clj$ is normalized. We now claim that $\clj \in E$.  
	For each $m \in \{1, \ldots, n\}$, we have
	\[
	\int_X g_m\, d\clj
	= \sum_{i=1}^n \sum_{j \in \Lambda_i} \cli(A_{ij})(g_m(x_{ij}))
	= \sum_{j \in \Lambda_m} \cli(A_{mj})(a_{mj})
	= \int_X g_m\, d\cli.
	\]
	Then, for each $i = 1, \ldots, n$, we estimate:
	\[
	\begin{split}
		\left|\int_X f_i\, d\clj_{h_i,k_i} - \int_X f_i\, d\cli_{h_i,k_i}\right|
		&\leq  \left|\int_X f_i d\clj_{h_i,k_i}-\int_X g_i d\clj_{h_i,k_i}\right|
		+\left|\int_X g_i  d\clj_{h_i,k_i}-\int_X g_i d\cli_{h_i,k_i}\right|\\
		&\quad\quad+\left|\int_X g_i d\cli_{h_i,k_i}-\int_X f_i d\cli_{h_i,k_i}\right|\\
		&\leq \left(\sup_{x \in X } \|f_i(x) - g_i(x)\|\right)
		\left(\|\clj_{h_i,k_i}\|(X) + \|\cli_{h_i,k_i}\|(X)\right) \\
		&\leq \frac{\epsilon}{2M} \cdot 2\|h_i\|\,\|k_i\| < \epsilon.
	\end{split}
	\]
	We have used the fact that
	$\|\clj_{h_i,k_i}\|(X), \|\cli_{h_i,k_i}\|(X) \leq \|h_i\|\,\|k_i\|$, which follows directly from the definitions of $\cli_{h_i,k_i},\clj_{h_i,k_i}$.
	Hence $\clj \in E$, completing the proof.
\end{proof}

Below, we present a Krein–Milman type theorem for the space of normalized instruments endowed with the BW topology. It is important to emphasize that, unlike the classical setting, compactness of  $I_\clh(X,\cla)$ is not required for the result.

\begin{thm}\label{Krein-Milman type theorem for instruments}\index{Krein-Milman type Theorem!for instruments}
	Let $X$ be a Hausdorff topological space and let $\cla$ and $\clh$ be separable $\cst$-algebra and separable Hilbert space, respectively. Then the $C^*$-convex hull of $C^*$-extreme points is
	dense in $I_{\clh}(X,\cla)$ with respect to the BW topology.
\end{thm}
\begin{proof}
	Fix an instrument $\cli\in I_{\clh}(X,\cla)$. By Proposition \ref{prop:normalized instruments with
		finite support are dense}, there exists a net $\cli_i\in I_{\clh}(X,\cla)$ such that
	$\cli_i\to\cli$ in $I_{\clh}(X,\cla)$ in BW topology and each $\cli_i$ is concentrated on a finite
	subset of $X$. Hence, it suffices to show that any instrument concentrated on a finite subset belongs to the $\cst$-convex hull. Without loss of generality, assume that $\cli$ is concentrated on a finite subset, say $\{x_1,\ldots,x_n\}$. We may therefore identify $X$ with this finite set. By the correspondence between completely positive maps and instruments (see Theorem~\ref{CP and instrument correspondence}), the instrument $\cli$ corresponds to a unital completely positive (UCP) map $\Phi_\cli:\oplus_{i=1}^n\cla\to\bh.$ According to the Krein–Milman type theorem for UCP maps \cite[Theorem~5.3]{manishjfa}, there exists a net $\Phi_j$ of $C^*$ extreme UCP maps such that $\Phi_j\to\Phi_\cli.$ The corresponding instruments
	$\cli_{\Phi_j}$ are then $\cst$-extreme points of $I_\clh(X,\cla),$ and as discussed in Proposition~\ref{BW topolgy are homeomorphic for regular instruments and UCP maps}, the convergence $\Phi_j\to\Phi_\cli$ implies $\cli_{\Phi_j}\to\cli$ in BW toplogy.
\end{proof}

\section{Tensor product approach}

The vector measure approach developed in the previous section is not the only way to define integration with respect to an instrument. Since bounded measurable functions naturally form a von Neumann algebra, it is equally natural to formulate the theory through completely positive maps on tensor product algebras. In this section we develop such an operator-algebraic approach, which extends the class of admissible integrands beyond the separably valued functions considered earlier.

Throughout this section, let $\clh$ be a separable Hilbert space. For a $\sigma$-finite measure $\nu$ on $(X,\ox)$, define
$
Ins_{\clh}(X,\cla)_\nu
=
\left\{
\cli:\ox\to CP(\cla,\bh):
(\mu_\cli)_\rho\ll\nu,\ \forall\,\rho\in\cls(\clh)
\right\},$
where
\[
(\mu_\cli)_\rho(A)=\operatorname{tr}\!\left(\rho\,\mu_\cli(A)\right),
\qquad A\in\ox.\]
Clearly,
$
Ins_{\clh}(X,\cla)
=
\bigcup_{\nu}
Ins_{\clh}(X,\cla)_\nu,$
where the union is taken over all $\sigma$-finite measures on $(X,\ox)$.

Now let $\cli\in Ins_{\clh}(X,\cla)_\nu$, and let
$(\clk,\pi,E,V)$ be its minimal bi-dilation (Theorem~\ref{Bi-dilation Theorem}), so that
$
\cli(A,a)=V^*\pi(a)E(A)V.$
Since $E$ is a spectral measure, the functional calculus yields a normal $*$-homomorphism
$
\psi_E:L^\infty(X,\nu)\rightarrow\bk,
~
\psi_E(f)=\int_X f\,dE.$
Because $\pi(\cla)\subseteq\psi_E(L^\infty(X,\nu))'$, the maps $\pi$ and $\psi_E$ induce a completely positive map
$
\Phi:
L^\infty(X,\nu)\otimes_{\max}\cla
\longrightarrow
\bk,
\Phi(f\otimes a)=\pi(a)\psi_E(f).
$
Consequently,
$
\Phi_\cli=V^*\Phi V:
L^\infty(X,\nu)\otimes_{\max}\cla
\longrightarrow
\bh$ is a completely positive map naturally associated with the instrument $\cli$.

When $\cla$ is a von Neumann algebra with separable predual and $\cli$ is not only normal, the map $\Phi_\cli$ extends to a normal completely positive map
$
\Phi_\cli:
L^\infty(X,\nu)\,\overline{\otimes}\,\cla
\longrightarrow
\bh.$
Using the canonical identification
$
L^\infty(X,\nu)\,\overline{\otimes}\,\cla
\cong
L^\infty(X,\cla,\nu)
$(see \cite[Theorem~1.22.13]{sakai}), we define
\[
\int_X f\,d\cli:=\Phi_\cli(f),
\qquad
f\in L^\infty(X,\cla,\nu).
\]

In what follows, we present a concrete example of an instrument that naturally arises from a completely positive map. Although this example was not originally framed in the language of instruments, it fits seamlessly within the framework introduced by Farenick, Plosker, Ramsey, and MacLaren \cite{plosker_ramsey_povmintegratiion_2,douglus_plosker_ramsey_povmintegration_1}, who developed a theory of integration with respect to POVM.
\begin{xmpl}\label{non_trivial_instrument}
Let $\mu:\ox\to\bh$ be a POVM, and let $\rho$ be a faithful trace-class operator such that the Radon--Nikodym derivative
$\frac{d\mu}{d\mu_\rho}$ exists. Define
$
\Phi_\mu:
L^\infty(X,\bh,\mu_\rho)\to\bh$
by
\[
\langle h,\Phi_\mu(f)k\rangle
=
\int_X
\operatorname{tr}
\!\left(
|h\rangle\langle k|
\sqrt{\frac{d\mu}{d\mu_\rho}}
\,f\,
\sqrt{\frac{d\mu}{d\mu_\rho}}
\right)
d\mu_\rho .
\]
By \cite[Theorem~2.8]{Plosker_Ramsey}, $\Phi_\mu$ is a normal completely positive map. Hence it induces the normal instrument
\[
\cli_\mu(A,a)=\Phi_\mu(a\,1_A),
\qquad
A\in\ox,\ a\in\bh,
\]
whose POVM marginal is precisely $\mu$.
\end{xmpl}

The tensor product formulation therefore defines the integral simply by
\[
\int_X f\,d\cli_\mu:=\Phi_\mu(f),
\qquad
f\in L^\infty(X,\bh,\mu_\rho).
\]
Unlike the Bartle-type construction developed earlier, this approach places no separability requirement on the essential range of the integrand. Consequently, it naturally accommodates a strictly larger class of operator-valued functions, providing a more flexible framework for integration with respect to normal instruments.

\section*{Acknowledgements}

The author would like to thank Dr.~Sruthymurali for several fruitful discussions during the early stages of this work, particularly those leading to Theorem~\ref{CP and instrument correspondence}. The author is deeply grateful to his supervisor, Prof.~B.~V.~Rajarama Bhat, for his constant encouragement and support, for carefully reading the manuscript, and for many valuable suggestions that significantly improved this work.
The author also gratefully acknowledges the financial support provided through Prof.~Bhat's J.~C.~Bose Fellowship (Fellowship No.~JBR/2021/000024). Finally, the author sincerely thanks the anonymous referee of his Ph.D. thesis for a careful reading of the manuscript, for identifying inaccuracies, and for insightful suggestions that substantially improved both the exposition and the mathematical content of this work.
        
        \bibliography{references}
	\bibliographystyle{plain}

\end{document}